\documentclass[reqno]{amsart}
\usepackage{mathrsfs}
\usepackage{hyperref}
\usepackage{float}
\usepackage{graphicx}

\usepackage{enumerate}
\newtheorem{theorem}{Theorem}[section]
\newtheorem{lemma}[theorem]{Lemma}

\theoremstyle{definition}
\newtheorem{definition}[theorem]{Definition}

\newcommand{\norm}[1]{\left\Vert#1\right\Vert}

\numberwithin{equation}{section}
\usepackage{amsfonts}

\usepackage{amsmath}
\usepackage{amssymb}
\usepackage{cite}
\usepackage{hyperref}
\hypersetup{
   colorlinks,
    citecolor=blue,
    filecolor=black,
    linkcolor=blue,
    urlcolor=magenta
}

\begin{document}
\font\nho=cmr10
\def\dive{\mathrm{div}}
\def\cal{\mathcal}
\def\L{\cal L}

\def \ud{\underline }
\def\id{{\indent }}
\def\f{\frac}
\def\non{{\noindent}}
 \def\le{\leqslant} 
 \def\leq{\leqslant}
 \def\geq{\geqslant} 
\def\rar{\rightarrow}
\def\Rar{\Rightarrow}
\def\ti{\times}
\def\i{\mathbb I}
\def\j{\mathbb J}
\def\si{\sigma}
\def\Ga{\Gamma}
\def\ga{\gamma}
\def\ld{{\lambda}}
\def\Si{\Psi}
\def\f{\mathbf F}
\def\r{\hro{R}}
\def\e{\cal{E}}
\def\B{\cal B}
\def\A{\mathcal{A}}
\def\p{\mathbb P}

\def\tet{\theta}
\def\Tet{\Theta}
\def\hro{\mathbb}
\def\ho{\mathcal}
\def\P{\ho P}
\def\E{\mathcal{E}}
\def\n{\mathbb{N}}
\def\M{\mathbb{M}}
\def\dMu{\mathbf{U}}
\def\dMcs{\mathbf{C}}
\def\dMcu{\mathbf{C^u}}
\def\vk{\vskip 0.2cm}
\def\td{\Leftrightarrow}
\def\df{\frac}
\def\Wei{\mathrm{We}}
\def\Rey{\mathrm{Re}}
\def\s{\mathbb S}
\def\l{\mathcal{L}}
\def\C+{C_+([t_0,\infty))}
\def\o{\cal O}

\title[AAP-mild Solutions of (P-P) Keller-Segel systems]{On asymptotically almost Periodic Solutions of a chemotaxis model on bounded domains}





\author[P.T. Xuan]{Pham Truong Xuan}
\address{Pham Truong Xuan\hfill\break 
Thang Long Institute of Mathematics and Applied Sciences (TIMAS), Thang Long University, Nghiem Xuan Yem, Hanoi, Vietnam.
} 
\email{xuanpt@thanglong.edu.vn}

\begin{abstract}  
In this paper, we investigate the existence, uniqueness, and exponential decay of asymptotically almost periodic (AAP-) mild solutions for the parabolic-parabolic Keller-Segel systems on a bounded domain $\Omega \subset \mathbb{R}^n$ with a smooth boundary. First, we establish the well-posedness of mild solutions for the corresponding linear systems by utilizing the dispersive and smoothing estimates of the Neumann heat semigroup on the bounded domain $\Omega$. We then prove the existence and uniqueness of AAP-mild solutions for the linear systems by providing a Massera-type principle. Next, using results of the linear systems and fixed-point arguments, we derive the well-posedness of such solutions for the Keller-Segel systems. Finally, the exponential decay of these solutions is demonstrated through a Gronwall-type inequality.
\end{abstract}



\maketitle

\tableofcontents

\section{Introduction}
{In the present paper we investigate} the parabolic-parabolic (P-P) Keller–Segel system on a bounded domain with smooth boundary $\Omega \subset \mathbb{R}^n \,\, (n \geqslant 2)$ that is described by the following equations
\begin{equation}\label{KSH} 
\left\{
  \begin{array}{rll}
u_t \!\! &= \Delta u - \nabla \cdot (u\nabla v) + g(t) \quad  & (t,x)\in \r_+ \times \Omega, \hfill \cr
v_t \!\!&= \Delta v - \gamma v + \kappa u + h(t) \quad & (t,x) \in \r_+ \times \Omega,\cr
\nabla u \cdot \nu \!\!&= \nabla v\cdot \nu =0 \quad & (t,x) \in \r_+ \times \partial\Omega,\cr
u(x,0) &= u_0(x) \quad &x\in \Omega,\cr
v(x,0) &= v_0(x) \quad &x\in \Omega,
\end{array}\right.
\end{equation}
where $\nu$ is the normal outer vector
on $\partial\Omega$, the functions $g(t,x)$ and $h(t,x)$ are given and the parameters $\gamma\geqslant 0$ and $\kappa \geqslant 0$ denote the decay and production rate of
the attractant, respectively. The unkowns of system \eqref{KSH} are $u(t,x)$ representing the density of cells and $v(t,x)$ which is the concentration of the chemoattractant.

If the function $h=0$ and the chemoattractant concentration is time independent, then one obtains a
simplified version of system \eqref{KSH} which is called the parabolic–elliptic Keller-Segel system and reads as
\begin{equation}\label{KSH'} 
\left\{
  \begin{array}{rll}
u_t \!\! &= \Delta u - \nabla \cdot (u\nabla v) + g(t) \quad  & (t,x)\in \r_+ \times \Omega, \hfill \cr
-\Delta v + \gamma v \!\!&= \kappa u \quad & (t,x) \in \r_+ \times \Omega,\cr
\nabla u \cdot \nu \!\!&= \nabla v\cdot \nu =0 \quad & (t,x) \in \r_+ \times \partial\Omega,\cr
u(x,0) &= u_0(x) \quad &x\in \Omega,\cr
v(x,0) &= v_0(x) \quad &x\in \Omega.
\end{array}\right.
\end{equation} 

The chemotaxis model introduced by Keller–Segel \cite{KeSe} which models aggregation of biological species. In particular, this model describes a situation where organisms move towards high concentration of food molecules or of a chemical secreted by themselves.
We refer the readers to some useful works concerning the well-posedness and long time behaviours of solutions for systems \eqref{KSH'} in {\cite{Bla,Chen2018,CoEs,Do,Ko1,Ko2,
Fe2021,Hi2020,Iwa2011} 
and references therein.}

We briefly recall some previous works on the Keller-Segel (P-P) systems and related systems on bounded domains with smooth boundaries in $\mathbb{R}^n$. Winkler \cite{Winker} provided the $L^p-L^q$-dispersive and smoothing estimates for the Neumann heat semigroup and employed these estimates to prove the exponential stability of mild solutions for the case $n \geq 3$. Additionally, Winkler studied finite-time blow-up in the parabolic-parabolic Keller–Segel system in higher dimensions in \cite{Win2}. In \cite{Cao}, Cao extended the previous estimates from \cite{Winker} to investigate smallness conditions on the initial data in optimal Lebesgue spaces, ensuring global boundedness and large time convergence for $n \geq 2$. Subsequently, Hao et al. \cite{Hao} provided global classical solutions to the Keller–Segel–Navier–Stokes system with matrix-valued sensitivity. Furthermore, Jiang established the global stability of Keller–Segel systems in critical Lebesgue spaces in \cite{Jiang2018} and later analyzed the global stability of homogeneous steady states in scaling-invariant spaces for a Keller–Segel–Navier–Stokes system in \cite{Jiang2020}. We also refer to other useful works, such as \cite{Win1,Win3}.

In addition, we refer the readers to some related works concerning the well-posedness and asymptotic behaviour of solutions for the parabolic-parabolic Keller-Segel systems on the whole space $\mathbb{R}^n \, (n \geqslant 2)$ in \cite{Bilet,Cal,Co2006,Sugi,Zhai,Nagai} and many references therein.

The well-posedness of mild solutions and their generalizations (such as periodic, almost periodic, asymptotically almost periodic, and pseudo almost periodic mild solutions, etc...) has been extensively studied for various parabolic and hyperbolic equations (see, for example, \cite{Dia} and references therein). In the context of chemotaxis models, we have established the existence, uniqueness, and stability of specific types of mild solutions for the parabolic-elliptic Keller-Segel systems: periodic mild solutions in \cite{Xuan2024}, almost periodic mild solutions in \cite{XuanLoan}, and asymptotically almost periodic mild solutions in \cite{XVT}. In related works on fluid dynamics \cite{XVQ, XV}, we also provided the well-posedness and exponential stability of asymptotically almost periodic (AAP-) mild solutions for the Navier-Stokes equations on real hyperbolic manifolds.
To the best of our knowledge, however, no existing work addresses the existence, uniqueness, and stability of AAP-mild solutions for the parabolic-parabolic (P-P) Keller-Segel system \eqref{KSH} on a bounded domain $\Omega \subset \mathbb{R}^n$.

In the present paper, {we address the above problems by proving first} the well-posedness of mild solutions for the linear system associated with \eqref{KSH}. Specifically, we establish well-posedness of mild solutions for the corresponding linear systems by using the $L^p-L^q$-dispersive and smoothing estimates of the Neumann heat semigroup (see Lemma \ref{Thm:linear}). We then define the solution operator and establish a Massera-type principle, demonstrating that this operator preserves the asymptotically almost periodic properties of given functions (see Theorem \ref{pest}), thereby showing that the linear system has a unique AAP-mild solution. Leveraging the well-posedness of the linear system and fixed-point arguments, we obtain the existence of AAP-mild solutions for the Keller-Segel (P-P) system \eqref{KSH} (see Theorem \ref{thm2.20} $(i)$). Finally, we establish the exponential stability of AAP- mild solutions using a Gronwall-type inequality for a double integral (see Theorem \ref{thm2.20} $(ii)$).

We note that, unlike the recent work \cite{XVT} on the existence and exponential stability of AAP mild solutions for the parabolic-elliptic Keller-Segel system, we encounter additional challenges in proving our main results due to the presence of two independent unknowns, $u$ and $v$, in system \eqref{KSH}. Moreover, since the scalar heat semigroup on real hyperbolic spaces also satisfies the $L^p-L^q$-dispersive and smoothing estimates (see for example \cite{Xuan2024}) with exponential decay rates, the results obtained in this paper appear applicable to the Keller-Segel (P-P) system in this framework as well. Together with the recent work \cite{XVT}, this paper provides comprehensive results on the global well-posedness and exponential stability of AAP-mild solutions for the Keller-Segel systems in two settings: the parabolic-parabolic system \eqref{KSH} and the parabolic-elliptic system \eqref{KSH'}. 

Our paper is organized as follows: in Section \ref{S2}, we recall some facts on the Keller-Segel (P-P) system, the $L^p-L^q$-dispersive and smoothing estimates of the Neumann heat semigroup and the concepts of almost periodic and asymptotically almost periodic functions; in Section \ref{S3} we give the proofs of well-posedness of mild solutions and AAP-mild solutions for the linear systems; in Section \ref{S4} we provided the well-posedness and exponential stability results of AAP-mild solutions for the Keller-Segel (P-P) system.

\section{The parabolic-parabolic Keller-Segel systems and concepts of functions} \label{S2}
For the sake of simplicity, we assume $\gamma=\kappa=1$, the Keller-Segel (P-P) system \eqref{KSH} on $\Omega\subset \mathbb{R}^n$ (where $n \geqslant 2$) becomes
\begin{equation}\label{KS} 
\left\{
  \begin{array}{rll}
u_t \!\! &= \Delta u - \nabla \cdot (u\nabla v) + g(t) \quad  & (t,x)\in \r_+ \times \Omega, \hfill \cr
v_t \!\!&= \Delta v - v + u + h(t) \quad & (t,x) \in \r_+ \times \Omega,\cr
\nabla u \cdot \nu \!\!&= \nabla v\cdot \nu =0 \quad & (t,x) \in \r_+ \times \partial\Omega,\cr
u(x,0) &= u_0(x) \quad &x\in \Omega,\cr
v(x,0) &= v_0(x) \quad &x\in \Omega,
\end{array}\right.
\end{equation}
where $\nu$ is the normal outer vector
on $\partial\Omega$. By employing the term of matrix, we can rewrite the system \eqref{KS} as follows
\begin{equation}\label{matrixKS} 
\left\{
  \begin{array}{rll}
\dfrac{\partial}{\partial t}\begin{bmatrix} u\\v\end{bmatrix} + \mathcal{A} \begin{bmatrix} u\\v\end{bmatrix} \!\! &= - \begin{bmatrix}\nabla \cdot (u\nabla v)\\0\end{bmatrix} + F\left(t,\begin{bmatrix}u\\v\end{bmatrix} \right) \quad  & (t,x)\in \r_+ \times \Omega, \hfill \cr
\nabla u \cdot \nu \!\!&= \nabla v\cdot \nu = 0 \quad & (t,x) \in \r_+ \times \partial\Omega,\cr
\begin{bmatrix}u(x,0) \\v(x,0)\end{bmatrix}&= \begin{bmatrix}u_0(x)\\v_0(x)\end{bmatrix} \quad &x\in \Omega,
\end{array}\right.
\end{equation}
where $\mathrm{Id}$ denotes identity operator and
\begin{equation}
\mathcal{A}= \begin{bmatrix}
-\Delta&&0\\
0&&-\Delta + \mathrm{Id}
\end{bmatrix} \hbox{   and   }F\left(t,\begin{bmatrix}u\\v\end{bmatrix} \right) = \begin{bmatrix}0\\ u \end{bmatrix} + \begin{bmatrix} g(t)\\h(t)\end{bmatrix}.
\end{equation}

We recall the dispersive and smoothing estimates of Nemann heat semigroup on the bounded domain $\Omega$ (see \cite{Cao,Winker}).
\begin{lemma}\label{Heatestimates}
Suppose $(e^{t\Delta })_{t>0}$ is the Neumann heat semigroup in $\Omega$, and let $\lambda_1 > 0$ denote the first nonzero eigenvalue of $ - \Delta$ 
in $\Omega$ under Neumann boundary conditions. Then there exist positive constants $k_1,k_2,k_3, k_4$ which only depend on $\Omega$ and we have the following estimates

\begin{itemize}
\item[$(i)$] If $1 \leq q \leq p\leq \infty$, then 
\begin{equation}\label{dispersive1}
\left\| e^{t {\Delta}}\omega\right\|_{L^p(\Omega)} \leq k_1(1+t^{- \frac{n}{2}(\frac{1}{q}-\frac{1}{p})}) e^{-\lambda_1t}\left\|\omega\right\|_{L^q(\Omega)} \text{for all }  t>0
\end{equation}
 holds for all $\omega \in L^q(\Omega)$ with $\int_{\Omega}\omega dx = 0$.

\item[$(ii)$] If $1 \leq q \leq p\leq \infty$, then 
\begin{equation}\label{dispersive2}
\left\| \nabla e^{t {\Delta}}\omega\right\|_{L^p(\Omega)} \leq k_2(1+t^{-\frac{1}{2}- \frac{n}{2}(\frac{1}{q}-\frac{1}{p})}) e^{-\lambda_1t}\left\|\omega\right\|_{L^q(\Omega)} \text{for all }  t>0
\end{equation}
 holds for all $\omega \in L^q(\Omega)$.

\item[$(iii)$] If $2 \leq q \leq p < \infty$, then 
\begin{equation}\label{dispersive3}
\left\| \nabla e^{t {\Delta}}\omega\right\|_{L^p(\Omega)} \leq k_3(1+t^{- \frac{n}{2}(\frac{1}{q}-\frac{1}{p})}) e^{-\lambda_1t}\left\|\nabla\omega\right\|_{L^q(\Omega)} \text{for all }  t>0
\end{equation}
 is true for all $\omega \in W^{1,p}(\Omega)$.

\item[$(iv)$] If $1 < q \leq p\leq \infty$, then 
\begin{equation}\label{dispersive4}
\left\| e^{t {\Delta}} \nabla \cdot \omega\right\|_{L^p(\Omega)} \leq k_4(1+t^{-\frac{1}{2}- \frac{n}{2}(\frac{1}{q}-\frac{1}{p})}) e^{-\lambda_1t}\left\|\omega\right\|_{L^q(\Omega)} \text{for all }  t>0
\end{equation}
 is valid for any $\omega \in (W^{1,p}(\Omega))^n$.
\end{itemize}
\end{lemma}
\begin{proof}
The proof was given in \cite[Lemma 2.1]{Cao} and \cite[Lemma 1.3]{Winker}.
\end{proof}
To serve the next sections we recall some notions of almost periodic and asymptotically almost periodic functions. {For a deeper understanding of these functions and their properties, we refer the reader to the references \cite{Bes,Dia,Fink,Gue,Kos1,Kos2,Lit}.}

{Let $X$ be a Banach space. We denote} 
$$C_b(\r, X):=\left\{f:\r \to X \mid f\hbox{ is continuous on $\r$ and }\sup_{t\in\r}\|f(t)\|_X<\infty\right\}$$
which is a Banach space endowed with the norm $\|f\|_{\infty, X}=\|f\|_{C_b(\r, X)}:=\sup\limits_{t\in\r}\|f(t)\|_X$.
Similarly, we denote 
$$C_b(\r_+, X):=\left\{f:\r_+ \to X \mid f\hbox{ is continuous on $\r_+$ and }\sup_{t\in\r_+}\|f(t)\|_X<\infty\right\}$$
which is also a  Banach space endowed with the norm $\|f\|_{\infty, X}=\|f\|_{C_b(\r_+, X)}:=\sup\limits_{t\in\r_+}\|f(t)\|_X$.

\begin{definition}
A function  $h \in C_b(\r, X )$ is called almost periodic function if for each $ \epsilon  > 0$, there exists $l_{\epsilon}>0 $ such that every interval of length $l_{\epsilon}$ contains at least a number $T $ with the following property
\begin{equation}
 \sup_{t \in \r } \| h(t+T)  - h(t) \| < \epsilon.
\end{equation}
The collection of all almost periodic functions $h:\r \to X $ will be denoted by $AP(\r,X)$ which is a Banach space endowed with the norm $\|h\|_{ AP(\r,X)}=\sup\limits_{t\in\r}\|h(t)\|_X.$
\end{definition}
To introduce the asymptotically almost periodic functions, we need the space  $C_0 (\r_+,X)$, that is, the collection of all continuous functions $\varphi: \r_+ \to X$ such that
 $$\mathop{\lim}\limits_{t \to \infty } \| \varphi(t) \|=0.$$
Clearly, $C_0 (\r_+,X)$  is a Banach space endowed with the norm $\|\varphi\|_{C_0 (\r_+,X)}=\sup\limits_{t\in\r_+}\|\varphi(t)\|_X$.
\begin{definition} 
A function  $f \in C_b(\r_+, X )$  is said to be (forward-) asymptotically almost periodic if there exist  $h \in AP(\r,X)$ and $ \varphi\in C_0(\r_+,X)$ such that
\begin{equation}
f(t) = h(t) + \varphi(t).
\end{equation}
We denote $AAP(\r_+, X):= \{f:\r_+ \to X \mid f\hbox{ is asymptotically almost periodic on $\r_+$}\}$. Under the norm $\|f\|_{ AAP(\r_+,X)}=\|h\|_{ AP(\r, X)}+\|\varphi\|_{ C_0(\r_+,X)}$, then $AAP(\r_+,X)$ is a Banach space.
\end{definition} 
The decomposition of asymptotically almost periodic functions is unique \rm( see \cite[Proposition 3.44, page 97]{Dia}), that is, 
$$AAP(\r_+,X) = AP(\r,X) \oplus C_0(\r_+,X).$$

The asymptotically almost periodic (AAP-) functions are genelizations of and almost periodic (AP-) and periodic functions. Precisely, we have the following inclusions
$$P(\r,X) \hookrightarrow AP(\r,X) \hookrightarrow AAP(\r_+,X) \hookrightarrow C_b(\r,X).$$
where $P(\r,X)$ is the space of all  continuous and periodic functions from $\r$ to $X$.\\
{\bf Example:}
\begin{itemize}
\item[$(i)$] The function $h(t)=\sin{ t}+\sin({\sqrt{2}t})$ is almost periodic but not periodic,  $\tilde{h}(t) =\sin{ t}+\sin({\sqrt{2}t})+ \dfrac{1}{|t|}$  is asymptotically almost periodic but not almost periodic.
\item[$(ii)$] Let $X$ be a Banach space and $g\in X - \left\{ 0\right\}$, we have that $f=hg\in AP(\r,X)$ and $\tilde{f}=\tilde{h}g \in AAP(\r_+,X)$ 
\end{itemize}

\section{Asymptotically almost periodic mild solutions of linear systems}\label{S3}
For a given vector $\begin{bmatrix}\omega\\\zeta\end{bmatrix}$ we study the following inhomogeneous linear equation corresponding to origin equation \eqref{matrixKS}: 
\begin{equation}\label{Linear_matrixKS1} 
\left\{
  \begin{array}{rll}
\dfrac{\partial}{\partial t}\begin{bmatrix} u\\v\end{bmatrix} + \mathcal{A} \begin{bmatrix} u\\v\end{bmatrix} \!\! &=  - \begin{bmatrix} \nabla \cdot(\omega\nabla \zeta) \\ 0 \end{bmatrix} + \begin{bmatrix}
g(t)\\ u(t) +h(t)
\end{bmatrix} \quad  & (t,x)\in \r_+ \times \Omega, \hfill \cr
\nabla u \cdot \nu \!\!&= \nabla v\cdot \nu = 0 \quad & (t,x) \in \r_+ \times \partial\Omega.\cr
\begin{bmatrix}u(x,0) \\v(x,0)\end{bmatrix}&= \begin{bmatrix}u_0(x)\\v_0(x)\end{bmatrix} \quad &x\in \Omega
\end{array}\right.
\end{equation}

By Duhamel's principle we define the mild solution $(u,v)$ of equation \eqref{Linear_matrixKS1} as the solution of following integral equations
\begin{equation}\label{interEq1}
u(t) = e^{t\Delta}u_0 - \int_0^t e^{(t-s)\Delta} \nabla \cdot (\omega\nabla\zeta)(s)ds + \int_0^t e^{(t-s)\Delta} g(s)ds 
\end{equation}
and 
\begin{equation}\label{interEq2}
v(t) = e^{t(\Delta-\mathrm{Id})}v_0 + \int_0^t e^{(t-s)(\mathrm{\Delta}-\mathrm{Id})}u(s)ds + \int_0^t e^{(t-s)(\mathrm{\Delta}-\mathrm{Id})}h(s)ds
\end{equation}
for $u$ satisfies \eqref{interEq1}.
In the matrix form, these equations are equivalent to
\begin{equation}\label{interEq}
\begin{bmatrix}
u\\v
\end{bmatrix}(t) = e^{-t\mathcal{A}}\begin{bmatrix}
u_0\\v_0
\end{bmatrix} + \mathcal{B}\left(\begin{bmatrix}
\omega\\\zeta
\end{bmatrix}\right)(s) + \mathcal{F}\left( \begin{bmatrix}
u\\0
\end{bmatrix} \right)(s),
\end{equation}
where
\begin{equation}
\mathcal{B}\left( \begin{bmatrix}\omega\\\zeta\end{bmatrix}\right)(s) = -\int_0^t e^{-(t-s)\mathcal{A}}\begin{bmatrix} \nabla \cdot(\omega\nabla \zeta) \\ 0 \end{bmatrix}(s)  ds
\end{equation}
and
\begin{equation}
\mathcal{F}\left( \begin{bmatrix} u \\0\end{bmatrix}\right)(s) = \int_0^t e^{-(t-s)\mathcal{A}}F\left( s,\begin{bmatrix}u \\0 \end{bmatrix} \right)ds = \int_0^t e^{-(t-s)\mathcal{A}} \begin{bmatrix}
g \\ u + h
\end{bmatrix}(s)  ds.
\end{equation}

Let $2\leqslant n < p$. We consider the well-posedness of system \eqref{Linear_matrixKS1} in the following Banach space
\begin{eqnarray*}
\mathcal{X} &=& \left\{ (u,v) \in C_b(\mathbb{R}_+, L^{\frac{p}{2}}(\Omega)\times L^{\frac{p}{2}}(\Omega)),\, \nabla v \in C_b(\mathbb{R}_+,L^{p}(\Omega)) \hbox{   such that}\right.\cr
&&\left. \hbox{  the function    $t\mapsto \| u(t)\|_{L^\frac{p}{2}} + \| v(t)\|_{L^{\frac{p}{2}}} + \| \nabla v(t)\|_{L^{p}} <+\infty$   belongs to $L^\infty(\mathbb{R}_+)$}\right\}
\end{eqnarray*}
endowed with the norm
$$\| (u,v)\|_{\mathcal{X}} = \sup_{t>0} \left( \| u(t)\|_{L^{\frac{p}{2}}} + \|v(t) \|_{L^{\frac{p}{2}}} + \| \nabla v(t)\|_{L^{p}}\right).$$

The existence and uniqueness of mild solution for the linear equation \eqref{Linear_matrixKS1} in the space $\mathcal{X}$ is established in the following lemma.
\begin{lemma}\label{Thm:linear}
Let $2\leqslant n$ and $\max\{3,n \} < p$. For the initial data $u_0\in C(\overline{\Omega})\cap L_0^1(\Omega)$ and $v_0 \in C^1(\overline{\Omega})\cap L_0^1(\Omega)$ such that $\begin{bmatrix}
u_0\\v_0
\end{bmatrix} \in L^{\frac{p}{2}}(\Omega)\times L^{\frac{p}{2}}(\Omega)$, $\nabla v_0 \in L^p(\Omega)$ and given functions $\begin{bmatrix}\omega\\\zeta\end{bmatrix} \in \mathcal{X}$ and $\begin{bmatrix}g\\h \end{bmatrix} \in C_b(\mathbb{R}_+, L^{\frac{p}{2}}(\Omega)\times L^{\frac{p}{2}}(\Omega))$, there exists a unique mild solution of linear equation \eqref{Linear_matrixKS1} satisfying the integral equation \eqref{interEq}. {Moreover, the following estimate holds}
\begin{equation}\label{boundedness12}
\norm{\begin{bmatrix}u\\ v\end{bmatrix}}_{\mathcal{X}} \leqslant C_1\norm{\begin{bmatrix}
u_0\\v_0
\end{bmatrix}}_{L^{\frac{p}{2}}\times L^{\frac{p}{2}}} + C_2\| \nabla v_0\|_{L^p} + C_3\norm{\begin{bmatrix}
\omega\\\zeta
\end{bmatrix}}^2_{\mathcal{X}} +  C_4 \norm{\begin{bmatrix}
g\\h
\end{bmatrix}}_{\infty,L^{\frac{p}{2}}\times L^{\frac{p}{2}}}.
\end{equation}
\end{lemma}
\begin{proof} 
Using the estimates \eqref{dispersive1} and \eqref{dispersive4} in Lemma \ref{Heatestimates} we have
\begin{eqnarray}\label{estu}
\norm{u(t)}_{L^{\frac{p}{2}}} &\leqslant& \norm{e^{t\Delta}u_0}_{L^{\frac{p}{2}}} + \int_0^t \norm{e^{(t-s)\Delta}\nabla\cdot(\omega\nabla\zeta)(s)}_{L^{\frac{p}{2}}}ds + \int_0^t \norm{e^{(t-s)\Delta}g(s)}_{L^{\frac{p}{2}}}ds\cr
&\leqslant& k_1e^{-\lambda_1 t}\norm{u_0}_{L^{\frac{p}{2}}} + k_4\int_0^t \left( 1 + (t-s)^{-\frac{1}{2}-\frac{n}{2p}} \right)  e^{-\lambda_1(t-s)} \norm{\omega(s)\nabla\zeta(s)}_{L^{\frac{p}{3}}}ds\cr
&&+ k_1\int_0^t e^{-\lambda_1 (t-s)}\norm{g(s)}_{L^{\frac{p}{2}}}ds\cr
&\leqslant& k_1\norm{u_0}_{L^{\frac{p}{2}}} + k_4 \norm{\omega}_{\infty,L^{\frac{p}{2}}}\norm{\nabla\zeta}_{\infty,L^{p}} \int_0^{+\infty} \left( 1 + z^{-\frac{1}{2}-\frac{n}{2p}} \right)  e^{-\lambda_1 z} dz\cr
&&+ k_1\norm{g}_{\infty,L^{\frac{p}{2}}}\int_0^{+\infty} e^{-\lambda_1 z}dz\cr
&\leqslant& k_1\norm{u_0}_{L^{\frac{p}{2}}} + k_4 C \norm{\omega}_{\infty,L^{\frac{p}{2}}}\norm{\nabla\zeta}_{\infty,L^{p}} + \frac{k_1}{\lambda_1}\norm{g}_{\infty,L^{\frac{p}{2}}},
\end{eqnarray}
where $C = \left( \frac{1}{\lambda_1} + \lambda_1^{-\frac{1}{2}+\frac{n}{2p}}{\bf \Gamma}\left( \frac{1}{2} - \frac{n}{2p} \right) \right) <+\infty$ since the boundedness of Gamma function ${\bf \Gamma}\left( \frac{1}{2} - \frac{n}{2p} \right)$ provided that $p>n$.

By the same way as above we can estimate
\begin{eqnarray}\label{estv}
\norm{v(t)}_{L^{\frac{p}{2}}} &\leqslant& \norm{e^{t(\Delta-\mathrm{Id})}v_0}_{L^{\frac{p}{2}}} + \int_0^t \norm{e^{(t-s)(\Delta-\mathrm{Id})}u(s)}_{L^{\frac{p}{2}}}ds + \int_0^t \norm{e^{(t-s)(\Delta-\mathrm{Id})}h(s)}_{L^{\frac{p}{2}}}ds\cr
&\leqslant& k_1 e^{-(\lambda_1+1)t} \norm{v_0}_{L^{\frac{p}{2}}} + k_1\int_0^t e^{-(\lambda_1+1)(t-s)} \norm{u(s)}_{L^{\frac{p}{2}}}ds \cr
&&+ k_1\int_0^t e^{-(\lambda_1+1)(t-s)}\norm{h(s)}_{L^{\frac{p}{2}}}ds\cr
&\leqslant& k_1 \norm{v_0}_{L^{\frac{p}{2}}} + \frac{k_1}{\lambda_1+1} \left(\| u\|_{\infty,L^{\frac{p}{2}}} + \norm{h}_{\infty,L^{\frac{p}{2}}}\right)\cr
&\leqslant& \max\left\{k_1,\frac{k_1^2}{\lambda_1+1} \right\} \norm{\begin{bmatrix}u_0\\v_0 \end{bmatrix}}_{L^{\frac{p}{2}}\times L^{\frac{p}{2}}} + \frac{k_1k_4C}{\lambda_1+1}\norm{\omega}_{\infty,L^{\frac{p}{2}}}\norm{\nabla\zeta}_{\infty,L^{p}} \cr
&&+ \max\left\{ \frac{k_1^2}{\lambda_1(\lambda_1+1)}, \frac{k_1}{\lambda_1+1}\right\} \norm{\begin{bmatrix}g\\h \end{bmatrix}}_{\infty,L^{\frac{p}{2}}\times L^{\frac{p}{2}}},
\end{eqnarray}
where the last estimate holds by using \eqref{estu}. 

Moreover, by using estimates \eqref{dispersive2} and \eqref{dispersive3} we have
\begin{eqnarray}\label{estderv}
\norm{\nabla v(t)}_{L^{p}} &\leqslant& \norm{\nabla e^{t(\Delta-\mathrm{Id})}v_0}_{L^{p}} + \int_0^t \norm{\nabla e^{(t-s)(\Delta-\mathrm{Id})} u(s)}_{L^{p}}ds + \int_0^t \norm{\nabla e^{(t-s)(\Delta-\mathrm{Id})}h(s)}_{L^{p}}ds\cr
&\leqslant& k_3 e^{-(\lambda_1+1)t} \norm{\nabla v_0}_{L^{p}} + k_2\int_0^t \left( 1 + (t-s)^{-\frac{1}{2}-\frac{n}{2p}} \right)e^{-(\lambda_1+1)(t-s)} \norm{u(s)}_{L^{\frac{p}{2}}}ds \cr
&&+ k_2\int_0^t \left( 1 + (t-s)^{-\frac{1}{2}-\frac{n}{2p}} \right) e^{-(\lambda_1+1)(t-s)}\norm{h(s)}_{L^{\frac{p}{2}}}ds \cr
&\leqslant& k_3 \norm{v_0}_{L^{\frac{p}{2}}} + k_2\left(\frac{1}{\lambda_1+1} + \lambda_1^{\frac{1}{2}-\frac{n}{2p}}{\bf \Gamma}\left( \frac{1}{2} - \frac{n}{2p} \right) \right)\left(\norm{u}_{\infty,L^{\frac{p}{2}}} + \norm{h}_{\infty,L^{\frac{p}{2}}}\right)\cr
&\leqslant& k_3 \norm{v_0}_{L^{\frac{p}{2}}} + k_2C''\left(\norm{u}_{\infty,L^{\frac{p}{2}}} + \norm{h}_{\infty,L^{\frac{p}{2}}}\right)\cr
&&\hbox{(where   } C''=\left(\frac{1}{\lambda_1+1} + \lambda_1^{-\frac{1}{2}+\frac{n}{2p}}{\bf \Gamma}\left( \frac{1}{2} - \frac{n}{2p} \right) \right))\cr
&\leqslant& \max\left\{ k_3, k_1k_2C''  \right\}\norm{\begin{bmatrix}
u_0\\v_0
\end{bmatrix}}_{L^{\frac{p}{2}}\times L^{\frac{p}{2}}} + k_2k_4C''C\norm{\omega}_{\infty,L^{\frac{p}{2}}}\norm{\nabla\zeta}_{\infty,L^p}\cr
&&+ \max\left\{ \frac{k_1k_2C''}{\lambda_1},k_2C'' \right\}\norm{\begin{bmatrix}
g\\h
\end{bmatrix}}_{L^{\frac{p}{2}}\times L^{\frac{p}{2}}},
\end{eqnarray}
where the last estimate holds by using again \eqref{estu}.

Combining inequalities \eqref{estu}, \eqref{estv} and \eqref{estderv} we obtain the boundedness \eqref{boundedness12} which leads to the existence of mild solution of \eqref{Linear_matrixKS1}. The uniqueness holds clearly.

\end{proof} 
As a direct consequence of Lemma \ref{Thm:linear}, we can define the solution operator $S: \mathcal{X}\to \mathcal{X}$ associating with linear equation \eqref{Linear_matrixKS1} as
\begin{align*}
S: \mathcal{X} &\rightarrow \mathcal{X}\cr
\begin{bmatrix}\omega\\\zeta\end{bmatrix}&\mapsto S\left(\begin{bmatrix}\omega\\\zeta\end{bmatrix}\right),
\end{align*}
where $S(\omega,\zeta)$ is a unique solution of integral equation \eqref{interEq}, i.e, mild solution of \eqref{Linear_matrixKS1}.

We state and prove the existence and uniqueness of AAP-mild solutions for linear equation \eqref{Linear_matrixKS1} in the following theorem.
\begin{theorem}\label{pest}
Let $2\leqslant n $ and $\max\{ 3,n\}<p$. For the initial data $u_0\in C(\overline{\Omega})\cap L_0^1(\Omega)$ and $v_0 \in C^1(\overline{\Omega})\cap L_0^1(\Omega)$ such that $\begin{bmatrix}
u_0\\v_0
\end{bmatrix} \in L^{\frac{p}{2}}(\Omega)\times L^{\frac{p}{2}}(\Omega)$, $\nabla v_0 \in L^p(\Omega)$ and a given AAP-function $t\mapsto (\omega,\zeta,g,h)(t)$ with respect on the norm 
$$|\norm{(\omega,\zeta,g,h)(t)}| = \norm{\omega(t)}_{L^{\frac{p}{2}}} + \norm{\zeta(t)}_{L^{\frac{p}{2}}} + \norm{\nabla\zeta(t)}_{L^p} + \norm{g(t)}_{L^{\frac{p}{2}}} + \norm{h(t)}_{L^{\frac{p}{2}}},$$
there exists a unique AAP-mild solution of linear equation \eqref{Linear_matrixKS1} satisfying the integral equation \eqref{interEq}.
\end{theorem} 
\begin{proof} 
This theorem is in fact a Massera-type principle for AAP-mild solutions of parabolic-parabolic Keller-Segel systems (see the similar subjects for parabolic-elliptic Keller-Segel systems in \cite{Xuan2024,XVT, XuanLoan}). In particular, we need to prove that the solution operator $S$ preserves the asymptotically almost periodic property of functions $\begin{bmatrix}
\omega\\\zeta \end{bmatrix}$ and $\begin{bmatrix}
g\\h\end{bmatrix}$. Indeed, we have
\begin{equation}\label{SO}
S\left(\begin{bmatrix}\omega\\\zeta\end{bmatrix}\right) = \begin{bmatrix}
u\\v
\end{bmatrix}
\end{equation}
which is a unique solution of integral equation \eqref{Linear_matrixKS1} which means that the functions $u$ and $v$ satisfy equations \eqref{interEq1} and \eqref{interEq2}, respectively.

By hypothesis, we can decompose
\begin{equation}\label{decom1}
\begin{bmatrix}
\omega\\\zeta\\g\\h \end{bmatrix} = \begin{bmatrix}
\omega_1\\\zeta_1\\g_1\\h_1 \end{bmatrix}+ \begin{bmatrix}
\omega_2\\\zeta_2\\g_2\\h_2 \end{bmatrix}
\end{equation}
for $(\omega_1,\zeta_1,g_1,h_1)$ is almost periodic in $C_b(\mathbb{R},L^{\frac{p}{2}}(\Omega)\times L^{\frac{p}{2}}(\Omega)\times L^{\frac{p}{2}}(\Omega)\times L^{\frac{p}{2}}(\Omega) )$ with respect to the norm
\begin{equation}\label{A}
\left|\norm{\begin{bmatrix}\omega_1\\\zeta_1\\g_1\\h_1\end{bmatrix}(t)}\right|= \norm{\omega_1(t)}_{L^{\frac{p}{2}}} + \norm{\zeta_1(t)}_{L^{\frac{p}{2}}} + \norm{\nabla\zeta_1(t)}_{L^p} + \norm{g_1(t)}_{L^{\frac{p}{2}}} + \norm{h_1(t)}_{L^{\frac{p}{2}}},\, t\in \mathbb{R}
\end{equation}
and $(\omega_2,\zeta_2,g_2,h_2)$ satifying
\begin{equation}\label{limitCD}
\lim_{t\to +\infty}\left|\norm{\begin{bmatrix}
\omega_2\\\zeta_2\\g_2\\h_2 \end{bmatrix}(t)}\right| = 0.
\end{equation}
Inserting \eqref{decom1} into \eqref{interEq1}, we obtain
\begin{eqnarray}\label{decom3}
u(t) &=& e^{t\Delta}u_0 - \int_0^t e^{(t-s)\Delta}\nabla \cdot[(\omega_1+\omega_2)\nabla(\zeta_1+\zeta_2)](s)ds + \int_0^t e^{(t-s)\Delta}(g_1+g_2)(s)ds\cr
&=& - \int_{-\infty}^te^{(t-s)\Delta}\nabla \cdot(\omega_1\nabla\zeta_1)(s)ds + \int_{-\infty}^t e^{(t-s)\Delta}g_1(s)ds\cr
&&+ e^{t\Delta}u_0 + \int_{-\infty}^0 e^{(t-s)\Delta}\nabla \cdot(\omega_1\nabla\zeta_1)(s)ds - \int_{-\infty}^0 e^{(t-s)\Delta}g_1(s)ds\cr
&& - \int_0^t e^{(t-s)\Delta}\nabla \cdot[\omega_1\nabla\zeta_2 + \omega_2\nabla\zeta](s)ds + \int_0^t e^{(t-s)\Delta}g_2(s)ds\cr
&=& u_1(t) + u_2(t),
\end{eqnarray}
where
\begin{eqnarray}
u_1(t) &=&- \int_{-\infty}^te^{(t-s)\Delta}\nabla \cdot(\omega_1\nabla\zeta_1)(s)ds + \int_{-\infty}^t e^{(t-s)\Delta}g_1(s)ds
\end{eqnarray}
and
\begin{eqnarray}\label{u2}
u_2(t) &=& e^{t\Delta}u_0 + \int_{-\infty}^0 e^{(t-s)\Delta}\nabla \cdot(\omega_1\nabla\zeta_1)(s)ds - \int_{-\infty}^0 e^{(t-s)\Delta}g_1(s)ds\cr
&& - \int_0^t e^{(t-s)\Delta}\nabla \cdot[\omega_1\nabla\zeta_2 + \omega_2\nabla\zeta](s)ds + \int_0^t e^{(t-s)\Delta}g_2(s)ds\cr
&=& e^{t\Delta}u_0 + \int_t^{+\infty} e^{z\Delta}\nabla \cdot(\omega_1\nabla\zeta_1)(t-z)dz - \int_t^{+\infty} e^{z\Delta}g_1(t-z)dz\cr
&& - \int_0^t e^{(t-s)\Delta}\nabla \cdot[\omega_1\nabla\zeta_2 + \omega_2\nabla\zeta](s)ds + \int_0^t e^{(t-s)\Delta}g_2(s)ds.
\end{eqnarray}
Inserting \eqref{decom1} and \eqref{decom3} into \eqref{interEq2}, we get
\begin{eqnarray}
v(t) &=& e^{t(\Delta-\mathrm{Id})}v_0 + \int_0^t e^{(t-s)(\Delta-\mathrm{Id})}(u_1+u_2)(s)ds + \int_0^t e^{(t-s)(\Delta-\mathrm{Id})}(h_1+h_2)(s)ds\cr
&=& \int_{-\infty}^t e^{(t-s)(\Delta-\mathrm{Id})}u_1(s)ds + \int_{-\infty}^t e^{(t-s)(\Delta-\mathrm{Id})}h_1(s)ds\cr
&& + e^{t(\Delta-\mathrm{Id})}v_0 - \int_{-\infty}^0 e^{(t-s)(\Delta-\mathrm{Id})}u_1(s)ds - \int_{-\infty}^0 e^{(t-s)(\Delta-\mathrm{Id})}h_1(s)ds\cr
&&+ \int_0^t e^{(t-s)(\Delta-\mathrm{Id})}u_2(s)ds + \int_0^t e^{(t-s)(\Delta-\mathrm{Id})}h_2(s)ds\cr
&=& v_1(t) + v_2(t),
\end{eqnarray}
where
\begin{equation}
v_1(t) = \int_{-\infty}^t e^{(t-s)(\Delta-\mathrm{Id})}u_1(s)ds + \int_{-\infty}^t e^{(t-s)(\Delta-\mathrm{Id})}h_1(s)ds
\end{equation}
and
\begin{eqnarray}\label{v2}
v_2(t) &=& e^{t(\Delta-\mathrm{Id})}v_0 - \int_{-\infty}^0 e^{(t-s)(\Delta-\mathrm{Id})}u_1(s)ds - \int_{-\infty}^0 e^{(t-s)(\Delta-\mathrm{Id})}h_1(s)ds\cr
&&+ \int_0^t e^{(t-s)(\Delta-\mathrm{Id})}u_2(s)ds + \int_0^t e^{(t-s)(\Delta-\mathrm{Id})}h_2(s)ds\cr
&=& e^{t(\Delta-\mathrm{Id})}v_0 - \int_t^{+\infty} e^{z(\Delta-\mathrm{Id})}u_1(t-z)dz - \int_t^{+\infty} e^{z(\Delta-\mathrm{Id})}h_1(t-z)dz\cr
&&+ \int_0^t e^{(t-s)(\Delta-\mathrm{Id})}u_2(s)ds + \int_0^t e^{(t-s)(\Delta-\mathrm{Id})}h_2(s)ds.
\end{eqnarray}

Now, we prove that the function $\begin{bmatrix}
u_1\\v_1 \end{bmatrix}$ is almost periodic in $\mathcal{X}$ with repsect to the norm 
\begin{equation}
\norm{\begin{bmatrix}
u\\v
\end{bmatrix}(t)} = \norm{u(t)}_{L^{\frac{p}{2}}} + \norm{v(t)}_{L^{\frac{p}{2}}} + \norm{\nabla v(t)}_{L^p},\, t\in \mathbb{R}
\end{equation}
and the function $\begin{bmatrix}
u_2\\v_2 \end{bmatrix}$ satisfying that
\begin{equation}\label{limitPAP}
\lim_{t\to +\infty}\norm{\begin{bmatrix}
u_2\\v_2 \end{bmatrix}(t)} = 0.
\end{equation}
Combining these with equation \eqref{SO}, we complete the proof. In particular, we establish the above properties by the following two steps:

{\bf Step 1:} Since the function  $(\omega_1,\zeta_1,g_1,h_1)$ is almost periodic in $C_b(\mathbb{R},L^{\frac{p}{2}}(\Omega)\times L^{\frac{p}{2}}(\Omega)\times L^{\frac{p}{2}}(\Omega)\times L^{\frac{p}{2}}(\Omega) )$ with respect to the norm \eqref{A}, we have that: for each $ \epsilon  > 0$, there exists $l_{\epsilon}>0 $ such that every interval of length $l_{\epsilon}$ contains at least a number $T $ satisfying
\begin{eqnarray}\label{almost1}
&&\norm{\omega_1(t+T)-\omega_1(t)}_{L^{\frac{p}{2}}} + \norm{\zeta_1(t+T)-\zeta_1(t)}_{L^{\frac{p}{2}}} + \norm{\nabla\zeta_1(t+T)-\nabla\zeta_1(t)}_{L^p} \cr
&&+ \norm{g_1(t+T)-g_1(t)}_{L^{\frac{p}{2}}} + \norm{h_1(t+T)-h_1(t)}_{L^{\frac{p}{2}}}<\epsilon,\, t\in \mathbb{R}.
\end{eqnarray}
Using \eqref{almost1} and the similar estimates as \eqref{estu} in the proof of Lemma \ref{Thm:linear} we can obtain that
\begin{eqnarray}\label{Apest1}
&&\norm{u_1(t+T)-u_1(t)}_{L^{\frac{p}{2}}}\cr
&\leqslant&  \norm{\int_{-\infty}^{t+T} e^{(t+T-s)\Delta}\nabla\cdot[(\omega_1\nabla\zeta_1)(s)]ds +\int_{-\infty}^{t} e^{(t-s)\Delta}\nabla\cdot[(\omega_1\nabla\zeta_1)(s)]ds }_{L^{\frac{p}{2}}} \cr
&&+ \norm{\int_{-\infty}^{t+T} e^{(t+T-s)\Delta}g_1(s)ds - \int_{-\infty}^{t} e^{(t-s)\Delta}g_1(s)ds}_{L^{\frac{p}{2}}}\cr &\leqslant&  \int_{-\infty}^t \norm{e^{(t-s)\Delta}\nabla\cdot[(\omega_1\nabla\zeta_1)(s+T)-(\omega_1\nabla\zeta_1)(s)]}_{L^{\frac{p}{2}}}ds \cr
&&+ \int_{-\infty}^t \norm{e^{(t-s)\Delta}(g_1(s+T)-g_1(s))}_{L^{\frac{p}{2}}}ds\cr
&\leqslant&  k_4\int_{-\infty}^t \left( 1 + (t-s)^{-\frac{1}{2}-\frac{n}{2p}} \right)  e^{-\lambda_1(t-s)} \left(\norm{(\omega_1(s+T)-\omega_1(s))\nabla\zeta(s+T)}_{L^{\frac{p}{3}}}  \right.\cr
&&\hspace{8cm}\left.+ \norm{\omega_1(s)\nabla(\zeta(s+T)-\zeta_1(s))}_{L^{\frac{p}{3}}} \right)ds\cr
&&+ k_1\int_{-\infty}^t e^{-\lambda_1 (t-s)}\norm{g(s+T)-g(s)}_{L^{\frac{p}{2}}}ds\cr
&\leqslant&  k_4 \norm{\omega_1(\cdot+T)-\omega_1(\cdot)}_{\infty,L^{\frac{p}{2}}}\norm{\nabla\zeta_1}_{\infty,L^{p}} \int_0^{+\infty} \left( 1 + z^{-\frac{1}{2}-\frac{n}{2p}} \right)  e^{-\lambda_1 z} dz\cr
&&+ k_4 \norm{\omega_1}_{\infty,L^{\frac{p}{2}}}\norm{\nabla\zeta_1(\cdot+T)-\nabla\zeta_1(\cdot)}_{\infty,L^{p}} \int_0^{+\infty} \left( 1 + z^{-\frac{1}{2}-\frac{n}{2p}} \right)  e^{-\lambda_1 z} dz\cr
&&+ k_1\norm{g(\cdot+T)-g(\cdot)}_{\infty,L^{\frac{p}{2}}}\int_0^{+\infty} e^{-\lambda_1 z}dz\cr
&\leqslant& k_4 C \left(\norm{\omega_1}_{\infty,L^{\frac{p}{2}}}+\norm{\nabla\zeta_1}_{\infty,L^{p}} + \frac{k_1}{\lambda_1} \right)\epsilon,
\end{eqnarray}
where the constant $C$ is given as in the proof of Lemma \ref{Thm:linear}. By the same way as \eqref{estv} in the proof of Lemma \ref{Thm:linear} we have
\begin{eqnarray}\label{Apest2}
&&\norm{v_1(t+T)-v_1(t)}_{L^{\frac{p}{2}}}\cr
&\leqslant& \norm{\int_{-\infty}^{t+T} e^{(t+T-s)(\Delta-\mathrm{Id})}u_1(s) - \int_{-\infty}^{t} e^{(t-s)(\Delta-\mathrm{Id})}u_1(s)}_{L^{\frac{p}{2}}}ds \cr
&&+ \norm{\int_{-\infty}^{t+T} e^{(t+T-s)(\Delta-\mathrm{Id})}h_1(s)ds - \int_{-\infty}^{t} e^{(t-s)(\Delta-\mathrm{Id})}h_1(s)ds }_{L^{\frac{p}{2}}}\cr
&\leqslant& \norm{\int_{-\infty}^{t} e^{(t-s)(\Delta-\mathrm{Id})}(u_1(s+T) - u_1(s))}_{L^{\frac{p}{2}}}ds + \norm{\int_{-\infty}^{t} e^{(t-s)(\Delta-\mathrm{Id})}(h_1(s+T)-h_1(s))ds}_{L^{\frac{p}{2}}}\cr
&\leqslant&  k_1\int_{-\infty}^t e^{-(\lambda_1+1)(t-s)} \norm{u_1(s+T)-u_1(s)}_{L^{\frac{p}{2}}}ds + k_1\int_{-\infty}^t e^{-(\lambda_1+1)(t-s)}\norm{h_1(s+T)-h_1(s)}_{L^{\frac{p}{2}}}ds\cr
&\leqslant& \frac{k_1}{\lambda_1+1} \left(\| u_1(\cdot+T)-u_1(\cdot)\|_{\infty,L^{\frac{p}{2}}} + \norm{h_1(\cdot+T)-h_1(\cdot)}_{\infty,L^{\frac{p}{2}}}\right)\cr
&\leqslant& \left\{\frac{k_1k_4C}{\lambda_1+1}\left(\norm{\omega_1}_{\infty,L^{\frac{p}{2}}}+\norm{\nabla\zeta_1}_{\infty,L^{p}}\right) +\frac{k_1}{\lambda_1+1} \right\}\epsilon,
\end{eqnarray}
{where the last holds} by using estimation \eqref{Apest1} and inequality \eqref{almost1}. Moreover, {we also have} that
\begin{eqnarray}\label{Apest3}
&&\norm{\nabla v_1(t+T)-\nabla v_1(t)}_{L^{p}} \cr
&\leqslant&  \norm{\int_{-\infty}^{t+T} \nabla e^{(t+T-s)(\Delta-\mathrm{Id})} u_1(s)ds - \int_{-\infty}^{t} \nabla e^{(t-s)(\Delta-\mathrm{Id})} u_1(s)ds}_{L^{p}}\cr
&& + \norm{ \int_{-\infty}^{t+T} \nabla e^{(t+T-s)(\Delta-\mathrm{Id})}h_1(s)ds - \int_{-\infty}^{t} \nabla e^{(t-s)(\Delta-\mathrm{Id})}h_1(s)ds}_{L^{p}}\cr
&\leqslant&  \int_{-\infty}^t \norm{\nabla e^{(t-s)(\Delta-\mathrm{Id})} (u_1(s+T)-u_1(s))}_{L^{p}}ds + \int_{-\infty}^t \norm{\nabla e^{(t-s)(\Delta-\mathrm{Id})}(h_1(s+T)-h_1(s))}_{L^{p}}ds\cr
&\leqslant& k_2\int_{-\infty}^t \left( 1 + (t-s)^{-\frac{1}{2}-\frac{n}{2p}} \right)e^{-(\lambda_1+1)(t-s)} \norm{u_1(s+T)-u_1(s)}_{L^{\frac{p}{2}}}ds \cr
&&+ k_2\int_{-\infty}^t \left( 1 + (t-s)^{-\frac{1}{2}-\frac{n}{2p}} \right) e^{-(\lambda_1+1)(t-s)}\norm{h_1(s+T)-h_1(s)}_{L^{\frac{p}{2}}}ds \cr
&\leqslant& k_2C''\left(\norm{u_1(\cdot+T)-u_1(\cdot)}_{\infty,L^{\frac{p}{2}}} + \norm{h_1(\cdot+T)-h_1(\cdot)}_{\infty,L^{\frac{p}{2}}}\right)\cr
&\leqslant& \left\{k_2k_4CC''\left(\norm{\omega_1}_{\infty,L^{\frac{p}{2}}}+\norm{\nabla\zeta_1}_{\infty,L^p} + \frac{k_1}{\lambda_1}\right) + k_2C''\right\}\epsilon.
\end{eqnarray}
Combining estimations \eqref{Apest1}, \eqref{Apest2} and \eqref{Apest3} we get
\begin{equation}
\norm{\begin{bmatrix}
u_1\\v_1
\end{bmatrix}(t+T)-\begin{bmatrix}
u_1\\v_1
\end{bmatrix}(t)} \leqslant \widetilde{C}\epsilon,\, t\in \mathbb{R},
\end{equation}
where
\begin{eqnarray}
\widetilde{C} &=& k_4 C \left(\norm{\omega_1}_{\infty,L^{\frac{p}{2}}}+\norm{\nabla\zeta_1}_{\infty,L^{p}} + \frac{k_1}{\lambda_1} \right)\cr
&&+ \frac{k_1k_4C}{\lambda_1+1}\left(\norm{\omega_1}_{\infty,L^{\frac{p}{2}}}+\norm{\nabla\zeta_1}_{\infty,L^{p}}\right) +\frac{k_1}{\lambda_1+1}\cr
&&+ k_2k_4CC''\left(\norm{\omega_1}_{\infty,L^{\frac{p}{2}}}+\norm{\nabla\zeta_1}_{\infty,L^p} + \frac{k_1}{\lambda_1}\right) + k_2C''.
\end{eqnarray}
Therefore, the function $\begin{bmatrix}u_1\\v_1\end{bmatrix}$ is almost periodic in $\mathcal{X}$.

{\bf Step 2:} We remain to prove the limit \eqref{limitPAP} which is equivalent to\
\begin{equation}
\lim_{t\to +\infty}\left( \norm{u_2(t)}_{L^{\frac{p}{2}}} + \norm{v_2(t)}_{L^{\frac{p}{2}}} + \norm{\nabla v_2(t)}_{L^p} \right) = 0.
\end{equation} 
Below, we prove in detail that $\lim\limits_{t\to +\infty}\norm{u_2(t)}_{L^{\frac{p}{2}}} =0$ by using the condition \eqref{limitCD}. The same limits of $v_2(t)$ and $\nabla v_2(t)$ hold by the same way amd we obtain the limit \eqref{limitPAP}. 

Now, by using formula \eqref{u2} we have
\begin{eqnarray}\label{u22}
\norm{u_2(t)}_{L^{\frac{p}{2}}} &\leqslant&  \norm{e^{t\Delta}u_0}_{L^{\frac{p}{2}}} + \int_t^{+\infty} \norm{e^{z\Delta}\nabla \cdot(\omega_1\nabla\zeta_1)(t-z)}_{L^{\frac{p}{2}}}dz + \int_t^{+\infty} \norm{e^{z\Delta}g_1(t-z)}_{L^{\frac{p}{2}}}dz\cr
&& + \int_0^t \norm{e^{(t-s)\Delta}\nabla \cdot[\omega_1\nabla\zeta_2 + \omega_2\nabla\zeta](s)}_{L^{\frac{p}{2}}} ds + \int_0^t \norm{e^{(t-s)\Delta}g_2(s)}_{L^{\frac{p}{2}}}ds.
\end{eqnarray}
Similar to estimate \eqref{estu} in the proof of Lemma \ref{Thm:linear}, we can establish that
\begin{eqnarray}\label{limitu2}
\norm{u_2(t)}_{L^{\frac{p}{2}}} &\leqslant&  k_1e^{-\lambda_1 t}\norm{u_0}_{L^{\frac{p}{2}}} + k_4\norm{\omega_1}_{\infty,L^{\frac{p}{2}}}\norm{\nabla\zeta_1}_{\infty,L^p}\int_t^{+\infty}\left( 1 + z^{-\frac{1}{2}-\frac{n}{2p}} \right)e^{-\lambda_1z}dz\cr
&&+ k_1\norm{g_1}_{\infty,L^{\frac{p}{2}}}\int_t^{+\infty}e^{-\lambda_1z}dz \cr
&&+ k_4\norm{\omega_1}_{\infty,L^{\frac{p}{2}}}\int_0^{+\infty}\left( 1 + z^{-\frac{1}{2}-\frac{n}{2p}} \right)e^{-\lambda_1z}\norm{\nabla\zeta_2(t-z)}_{L^p}dz\cr
&&+ k_4\norm{\nabla\zeta}_{\infty,L^p}\int_0^{+\infty}\left( 1 + z^{-\frac{1}{2}-\frac{n}{2p}} \right)e^{-\lambda_1z}\norm{\omega_2(t-z)}_{L^{\frac{p}{2}}}dz\cr
&&+ k_1\int_0^{+\infty}e^{-\lambda_1z} \norm{g_2(t-z)}_{L^{\frac{p}{2}}}dz.
\end{eqnarray}
Clearly, we have the limits of three first terms in the right-hand side of \eqref{limitu2} are equal to zero. To prove that the limits of remain terms are equal to zero, we need to use \eqref{limitCD}. For example, we prove that
\begin{equation}\label{limitzeta2}
\lim_{t\to +\infty}\int_0^{+\infty}\left( 1 + z^{-\frac{1}{2}-\frac{n}{2p}} \right)e^{-\lambda_1z}\norm{\nabla\zeta_2(t-z)}_{L^p}dz =0.
\end{equation}  
Indeed, since $\lim\limits_{t\to +\infty}\norm{\nabla\zeta_2(t)}_{L^p} =0$, we have for each $\varepsilon>0$, there exists $T_0$ large enough such that for all $t>T_0$, then
$\norm{\nabla\zeta_2(t)}_{L^p}<\varepsilon$. Therefore, for $t>2T_0$ we have
\begin{eqnarray}\label{limPAP1}
&&\int_0^{+\infty}\left( 1 + z^{-\frac{1}{2}-\frac{n}{2p}} \right)e^{-\lambda_1z}\norm{\nabla\zeta_2(t-z)}_{L^p}dz\cr
&=& \int_0^{t/2}\left( 1 + z^{-\frac{1}{2}-\frac{n}{2p}} \right)e^{-\lambda_1z}\norm{\nabla\zeta_2(t-z)}_{L^p}dz
\cr
&&+ \int_{t/2}^{+\infty}\left( 1 + z^{-\frac{1}{2}-\frac{n}{2p}} \right)e^{-\lambda_1z}\norm{\nabla\zeta_2(t-z)}_{L^p}dz\cr
&\leqslant& \varepsilon \int_0^{t/2}\left( 1 + z^{-\frac{1}{2}-\frac{n}{2p}} \right)e^{-\lambda_1z}dz + \norm{\nabla \zeta_2}_{\infty,L^p} \int_{t/2}^{+\infty}\left( 1 + z^{-\frac{1}{2}-\frac{n}{2p}} \right)e^{-\lambda_1z}dz\cr
&\leqslant& \varepsilon \int_0^{+\infty}\left( 1 + z^{-\frac{1}{2}-\frac{n}{2p}} \right)e^{-\lambda_1z}dz + \norm{\nabla \zeta_2}_{\infty,L^p} \int_{t/2}^{+\infty}\left( 1 + z^{-\frac{1}{2}-\frac{n}{2p}} \right)e^{-\lambda_1z}dz\cr
&\leqslant& \varepsilon \left(\frac{1}{\lambda_1} + \lambda_1^{-\frac{1}{2}+\frac{n}{2p}}{\bf \Gamma}\left( \frac{1}{2}-\frac{n}{2p} \right)\right)+ \norm{\nabla \zeta_2}_{\infty,L^p} \int_{t/2}^{+\infty}\left( 1 + z^{-\frac{1}{2}-\frac{n}{2p}} \right)e^{-\lambda_1z}dz.
\end{eqnarray}
From the convergence
\begin{equation}
\int_{t/2}^{+\infty}\left( 1 + z^{-\frac{1}{2}-\frac{n}{2p}} \right)e^{-\lambda_1z}dz<\frac{1}{\lambda_1} + \lambda_1^{-\frac{1}{2}+\frac{n}{2p}}{\bf \Gamma}\left( \frac{1}{2}-\frac{n}{2p}\right)  <+\infty, 
\end{equation}
we have
\begin{equation}
\lim_{t\to +\infty }\int_{t/2}^{+\infty}\left( 1 + z^{-\frac{1}{2}-\frac{n}{2p}} \right)e^{-\lambda_1z}dz = 0.
\end{equation}
Hence, there exists $T_1$ large enough such that 
\begin{equation}
\int_{t/2}^{+\infty}\left( 1 + z^{-\frac{1}{2}-\frac{n}{2p}} \right)e^{-\lambda_1z}dz<\varepsilon
\end{equation}
for all $t>T_1$. Plugging this into \eqref{limPAP1}, we have for all $t\geqslant \max\{ 2T_0,T_1\}$:
\begin{eqnarray}
&&\int_0^{+\infty}\left( 1 + z^{-\frac{1}{2}-\frac{n}{2p}} \right)e^{-\lambda_1z}\norm{\nabla\zeta_2(t-z)}_{L^p}dz\cr
&\leqslant& \varepsilon \left(\frac{1}{\lambda_1} + \lambda_1^{-\frac{1}{2}+\frac{n}{2p}}{\bf \Gamma}\left( \frac{1}{2}-\frac{n}{2p} \right) + \norm{\nabla \zeta_2}_{\infty,L^p}\right).
\end{eqnarray}
This leads to limit \eqref{limitzeta2}. By the same way we get the limits of two last terms in the right-hand side of \eqref{u22} are equal to zero and our proof is complete.

\end{proof}

\section{Semi-linear systems: the existence and exponential stability} \label{S4} 
In this section, we establish the existence, uniqueness and exponential stability of AAP-mild solutions for the semi-linear system \eqref{KS}. Similar to the linear system, we define the mild solutions of semi-linear system \eqref{KS} by the solutions of the following integral equations
\begin{equation}\label{interEq11}
u(t) = e^{t\Delta}u_0 - \int_0^t e^{(t-s)\Delta} \nabla \cdot (u\nabla v)(s)ds + \int_0^t e^{(t-s)\Delta} g(s)ds 
\end{equation}
and 
\begin{equation}\label{interEq22}
v(t) = e^{t(\Delta-\mathrm{Id})}v_0 + \int_0^t e^{(t-s)(\mathrm{\Delta}-\mathrm{Id})}u(s)ds + \int_0^t e^{(t-s)(\mathrm{\Delta}-\mathrm{Id})}h(s)ds
\end{equation}
for $u$ satisfies \eqref{interEq1}.
In the matrix form, these equations are equivalent to
\begin{equation}\label{interEq'}
\begin{bmatrix}
u\\v
\end{bmatrix}(t) = e^{-t\mathcal{A}}\begin{bmatrix}
u_0\\v_0
\end{bmatrix} + \mathcal{B}\left(\begin{bmatrix}
u\\v
\end{bmatrix}\right)(s) + \mathcal{F}\left( \begin{bmatrix}
u\\0
\end{bmatrix} \right)(s),
\end{equation}
where
\begin{equation}
\mathcal{B}\left( \begin{bmatrix}u\\v\end{bmatrix}\right)(s) = -\int_0^t e^{-(t-s)\mathcal{A}}\begin{bmatrix} \nabla \cdot(u\nabla v) \\ 0 \end{bmatrix}(s)  ds
\end{equation}
and
\begin{equation}
\mathcal{F}\left( \begin{bmatrix} u \\0\end{bmatrix}\right)(s) = \int_0^t e^{-(t-s)\mathcal{A}}F\left( s,\begin{bmatrix}u \\0 \end{bmatrix} \right)ds = \int_0^t e^{-(t-s)\mathcal{A}} \begin{bmatrix}
g \\ u + h
\end{bmatrix}(s)  ds.
\end{equation}
We state and prove the main results of this section in the following theorem. 
\begin{theorem}\label{thm2.20}
Let $2\leqslant n$ and $\max\{3,n \} < p$. For the initial data $u_0\in C(\overline{\Omega})\cap L_0^1(\Omega)$ and $v_0 \in C^1(\overline{\Omega})\cap L_0^1(\Omega)$ such that $\begin{bmatrix}
u_0\\v_0
\end{bmatrix} \in L^{\frac{p}{2}}(\Omega)\times L^{\frac{p}{2}}(\Omega)$, $\nabla v_0 \in L^p(\Omega)$ and a given AAP-function $\begin{bmatrix}g\\h \end{bmatrix} \in C_b(\mathbb{R}_+, L^{\frac{p}{2}}(\Omega)\times L^{\frac{p}{2}}(\Omega))$. If the norms $\norm{\begin{bmatrix} u_0\\v_0\end{bmatrix}}$ and $\norm{\begin{bmatrix}
g\\h
\end{bmatrix}}_{L^{\frac{p}{2}}\times L^{\frac{p}{2}}}$ are small enough, then the following assertions hold:
\begin{itemize}
\item[$(i)$] There exists a unique AAP-mild solution $(\hat{u},\hat{v})$ of Keller-Segel (P-P) system \eqref{KS} satisfying integral equations \eqref{interEq11} and \eqref{interEq22}. 

\item[$(ii)$] The mild solution $(\hat{u},\hat{v})$ is exponential stable in the sense that: for another bounded mild solution $(\tilde{u},\tilde{v}) $ such that $\norm{(u_0,v_0)-(\tilde{u}(0),\tilde{v}(0))}$ small enough, we have
\begin{equation}\label{expStable}
\norm{\begin{bmatrix}
\hat{u}-\tilde{u}\\\hat{v}-\tilde{v}
\end{bmatrix}(t)} \leqslant De^{-\sigma t}\norm{\begin{bmatrix}
u_0-\tilde{u}(0)\\v_0-\tilde{v}(0)
\end{bmatrix}},
\end{equation}
where $0<\sigma<\lambda_1$ and $D$ is a positive constant which does not depend on $(\hat{u},\hat{v})$ and $(\tilde{u},\tilde{v})$. Here, we recall that
\begin{equation*}
\norm{\begin{bmatrix}
\omega\\\zeta
\end{bmatrix}(t)} = \norm{\omega(t)}_{L^{\frac{p}{2}}} + \norm{\zeta(t)}_{L^{\frac{p}{2}}} + \norm{\nabla \zeta(t)}_{L^p}.
\end{equation*}
\end{itemize}
\end{theorem} 
\begin{proof}
$(i)$ We first establish the well-posedness. We denote $B^{AAP}_{\rho}$ which is a ball centered at $0$ and radius $\rho>0$ and consists of all AAP-function in $\mathcal{X}$. For each given function $\begin{bmatrix}
\omega\\\zeta
\end{bmatrix} \in B_{\rho}^{AAP}$ we consider the following linear integral equation
\begin{equation}\label{eq}
\begin{bmatrix}
u\\v
\end{bmatrix}(t) = e^{-t\mathcal{A}}\begin{bmatrix}
u_0\\v_0
\end{bmatrix} + \mathcal{B}\left(\begin{bmatrix}
\omega\\\zeta
\end{bmatrix}\right)(s) + \mathcal{F}\left( \begin{bmatrix}
u\\0
\end{bmatrix} \right)(s).
\end{equation} 
By Theorem \ref{pest}, this integral has a unique solution $\begin{bmatrix}
u\\v
\end{bmatrix}$ and we can set a mapping $\Phi: B_\rho^{AAP} \to B_\rho^{AAP}$ by
\begin{equation}
\Phi\left( \begin{bmatrix}
\omega\\\zeta
\end{bmatrix} \right) = \begin{bmatrix}
u\\v
\end{bmatrix}
\end{equation}
which is a unique solution of \eqref{eq}.
Now, we prove that $\Phi$ maps $B_\rho^{AAP}$ into itself and is contraction. Indeed, by using Lemma \ref{Thm:linear} we have
\begin{eqnarray}
\norm{\Phi\left( \begin{bmatrix}
\omega\\\zeta
\end{bmatrix} \right) (t)} &\leqslant& \max\left\{ C_1,C_2 \right\}\norm{\begin{bmatrix} u_0\\v_0\end{bmatrix}} + C_2\norm{\begin{bmatrix}
\omega\\\zeta
\end{bmatrix}}_{\mathcal{X}}^2 + C_4\norm{\begin{bmatrix}
g\\h
\end{bmatrix}}_{L^{\frac{p}{2}}\times L^{\frac{p}{2}}}\cr
&\leqslant& \max\left\{ C_1,C_2 \right\}\norm{\begin{bmatrix} u_0\\v_0\end{bmatrix}} + C_2\rho^2 + C_4\norm{\begin{bmatrix}
g\\h
\end{bmatrix}}_{L^{\frac{p}{2}}\times L^{\frac{p}{2}}}\cr
&\leqslant& \rho
\end{eqnarray} 
provided that $\rho$, $\norm{\begin{bmatrix} u_0\\v_0\end{bmatrix}}$ and $\norm{\begin{bmatrix}
g\\h
\end{bmatrix}}_{L^{\frac{p}{2}}\times L^{\frac{p}{2}}}$ are small enough. Hence, the mapping $\Phi$ map $B_\rho^{AAP}$ into itself.

Now, for given functions $\begin{bmatrix}
\omega_1\\\zeta_1
\end{bmatrix}, \, \begin{bmatrix}
\omega_2\\\zeta_2
\end{bmatrix} \in B_{\rho}^{AAP}$ we have
\begin{eqnarray}\label{phi-phi}
\Phi\left( \begin{bmatrix}
\omega_1\\\zeta_1
\end{bmatrix} \right) - \Phi\left( \begin{bmatrix}
\omega_2\\\zeta_2
\end{bmatrix} \right) &=& \mathcal{B}\left(\begin{bmatrix}
\omega_1\\\zeta_1
\end{bmatrix}\right)(s) - \mathcal{B}\left(\begin{bmatrix}
\omega_2\\\zeta_2
\end{bmatrix}\right)(s)+ \mathcal{F}\left( \begin{bmatrix}
u_1\\0
\end{bmatrix} \right)(s) - \mathcal{F}\left( \begin{bmatrix}
u_2\\0
\end{bmatrix} \right)(s)\cr
&=& \int_0^t e^{-(t-s)\mathcal{A}}\begin{bmatrix} \nabla \cdot(-\omega_1\nabla \zeta_1 + \omega_2\nabla\zeta_2) \\ 0 \end{bmatrix}(s)  ds \cr
&&+ \int_0^t e^{-(t-s)\mathcal{A}} \begin{bmatrix}
0 \\ u_1 - u_2
\end{bmatrix}(s)  ds\cr
&=& \int_0^t e^{-(t-s)\mathcal{A}}\begin{bmatrix} \nabla \cdot[-(\omega_1-\omega_2)\nabla \zeta_1 + \omega_2\nabla(-\zeta_1+\zeta_2)] \\ 0 \end{bmatrix}(s)  ds \cr
&&+ \int_0^t e^{-(t-s)\mathcal{A}} \begin{bmatrix}
0 \\ u_1 - u_2
\end{bmatrix}(s)  ds\cr
&=& \int_0^t e^{-(t-s)\mathcal{A}}\begin{bmatrix} \nabla \cdot[-(\omega_1-\omega_2)\nabla \zeta_1] \\ 0 \end{bmatrix}(s)  ds \cr
&&+ \int_0^t e^{-(t-s)\mathcal{A}}\begin{bmatrix} \nabla \cdot[\omega_2\nabla (-\zeta_1+\zeta_2)] \\ 0 \end{bmatrix}(s)  ds \cr
&&+ \int_0^t e^{-(t-s)\mathcal{A}} \begin{bmatrix}
0 \\ u_1 - u_2
\end{bmatrix}(s)  ds,
\end{eqnarray}
where we set
\begin{equation}
\Phi\left( \begin{bmatrix}
\omega_1\\\zeta_1
\end{bmatrix} \right) = \begin{bmatrix}
u_1\\v_1
\end{bmatrix}\hbox{   and   } \Phi\left( \begin{bmatrix}
\omega_2\\\zeta_2
\end{bmatrix} \right) = \begin{bmatrix}
u_2\\v_2
\end{bmatrix}.
\end{equation}
By the same way as in the proof of Lemma \ref{Thm:linear} we can estimate that
\begin{eqnarray}\label{core}
\norm{\Phi\left( \begin{bmatrix}
\omega_1\\\zeta_1
\end{bmatrix} \right) - \Phi\left( \begin{bmatrix}
\omega_2\\\zeta_2
\end{bmatrix} \right)}_{\mathcal{X}} &\leq& 2C_3 \left( \norm{\begin{bmatrix}
\omega_1\\\zeta_1
\end{bmatrix}+\begin{bmatrix}
\omega_2\\\zeta_2
\end{bmatrix}}_{\mathcal{X}}  \right) \left(\norm{\begin{bmatrix}
\omega_1\\\zeta_1
\end{bmatrix}-\begin{bmatrix}
\omega_2\\\zeta_2
\end{bmatrix}}_{\mathcal{X}}\right)\cr
&\leqslant& 4C_3\rho \left(\norm{\begin{bmatrix}
\omega_1\\\zeta_1
\end{bmatrix}-\begin{bmatrix}
\omega_2\\\zeta_2
\end{bmatrix}}_{\mathcal{X}}\right).
\end{eqnarray}
Hence, the mapping $\Phi$ is contraction provided that $\rho<\dfrac{1}{4C_3}$.

By using fixed point arguments, there is a  fixed point $\begin{bmatrix}
\hat{u}\\\hat{v}
\end{bmatrix}$ of $\Phi$ which is clearly solution of integral equation \eqref{interEq'}. Therefore, the Keller-Segel system \eqref{KS} has a mild solution. The uniqueness holds by using core estimate \eqref{core}. 

$(ii)$ Now, we prove the exponential stability \eqref{expStable} of mild solution $\begin{bmatrix}
\hat{u}\\\hat{v}
\end{bmatrix}$. In particular, there exists a positive constant $\tilde{\rho}>0$ such that $\norm{\begin{bmatrix}
\tilde{u}\\\tilde{v}
\end{bmatrix}}_{\mathcal{X}}\leqslant \tilde{\rho}$.

By the same way as \eqref{phi-phi} we can express
\begin{eqnarray}
\begin{bmatrix}
\hat{u}\\\hat{v}
\end{bmatrix}(t) - \begin{bmatrix}
\tilde{u}\\\tilde{v}
\end{bmatrix}(t) &=& e^{-t\mathcal{A}}\left(  \begin{bmatrix}
 u_0\\v_0
\end{bmatrix}- \begin{bmatrix}
 \tilde{u}(0)\\\tilde{v}(0)
\end{bmatrix} \right)\cr
&&+ \int_0^t e^{-(t-s)\mathcal{A}}\begin{bmatrix} \nabla \cdot[-(\hat{u}-\tilde{u})\nabla \hat{v}] \\ 0 \end{bmatrix}(s)  ds \cr
&&+ \int_0^t e^{-(t-s)\mathcal{A}}\begin{bmatrix} \nabla \cdot[\tilde{u}\nabla (-\hat{v}+\tilde{v})] \\ 0 \end{bmatrix}(s)  ds \cr
&&+ \int_0^t e^{-(t-s)\mathcal{A}} \begin{bmatrix}
0 \\ \hat{u} - \tilde{u}
\end{bmatrix}(s)  ds.
\end{eqnarray}
This leads to
\begin{eqnarray}\label{stab}
\norm{\begin{bmatrix}
\hat{u}\\\hat{v}
\end{bmatrix}(t) - \begin{bmatrix}
\tilde{u}\\\tilde{v}
\end{bmatrix}(t)} &\leqslant& \norm{e^{-t\mathcal{A}}\left(  \begin{bmatrix}
 u_0\\v_0
\end{bmatrix}- \begin{bmatrix}
 \tilde{u}(0)\\\tilde{v}(0)
\end{bmatrix} \right)}\cr
&&+ \int_0^t \norm{e^{-(t-s)\mathcal{A}}\begin{bmatrix} \nabla \cdot[-(\hat{u}-\tilde{u})\nabla \hat{v}] \\ 0 \end{bmatrix}(s)}  ds \cr
&&+ \int_0^t \norm{e^{-(t-s)\mathcal{A}}\begin{bmatrix} \nabla \cdot[\tilde{u}\nabla (-\hat{v}+\tilde{v})] \\ 0 \end{bmatrix}(s)}  ds \cr
&&+ \int_0^t \norm{e^{-(t-s)\mathcal{A}} \begin{bmatrix}
0 \\ \hat{u} - \tilde{u}
\end{bmatrix}(s)}  ds.
\end{eqnarray}
By the same way as \eqref{estu}, \eqref{estv} and \eqref{estderv} in the proof of Lemma \ref{Thm:linear} we can estimate that
\begin{eqnarray}
&&\norm{\hat{u}(t)-\tilde{u}(t)}_{L^{\frac{p}{2}}} \cr
&\leqslant& \norm{e^{t\Delta}(u_0-\tilde{u}(0))}_{L^{\frac{p}{2}}} + \int_0^t \norm{e^{(t-s)\Delta}\nabla\cdot[(\hat{u}-\tilde{u})\nabla\hat{v}](s)}_{L^{\frac{p}{2}}}ds \cr
&&+ \int_0^t \norm{e^{(t-s)\Delta}\nabla\cdot[\tilde{u}\nabla(-\hat{v}+\tilde{v})](s)}_{L^{\frac{p}{2}}}ds \cr
&\leqslant& k_1e^{-\lambda_1 t}\norm{u_0-\tilde{u}(0)}_{L^{\frac{p}{2}}} + k_4\int_0^t \left( 1 + (t-s)^{-\frac{1}{2}-\frac{n}{2p}} \right)  e^{-\lambda_1(t-s)} \norm{(\hat{u}(s)-\tilde{u}(s))\nabla\hat{v}(s)}_{L^{\frac{p}{3}}}ds\cr
&&+ k_4\int_0^t \left( 1 + (t-s)^{-\frac{1}{2}-\frac{n}{2p}} \right)  e^{-\lambda_1(t-s)} \norm{\tilde{u}(s)\nabla(-\hat{v}(s)+\tilde{v}(s))}_{L^{\frac{p}{3}}}ds\cr
&\leqslant& k_1e^{-\lambda_1 t}\norm{u_0-\tilde{u}(0)}_{L^{\frac{p}{2}}} + k_4 \norm{\nabla\hat{v}}_{\infty,L^{p}} \int_0^t \left( 1 + (t-s)^{-\frac{1}{2}-\frac{n}{2p}} \right)  e^{-\lambda_1 (t-s)}\norm{\hat{u}(s)-\tilde{u}(s)}_{L^{\frac{p}{2}}} ds\cr
&&+ k_4 \norm{\tilde{u}}_{\infty,L^{\frac{p}{2}}} \int_0^t \left( 1 + (t-s)^{-\frac{1}{2}-\frac{n}{2p}} \right)  e^{-\lambda_1 (t-s)}\norm{\nabla(-\hat{v}(s)+\tilde{v}(s))}_{L^p} ds\cr
&\leqslant& k_1e^{-\lambda_1 t}\norm{u_0-\tilde{u}(0)}_{L^{\frac{p}{2}}} + k_4 \rho \int_0^t \left( 1 + (t-s)^{-\frac{1}{2}-\frac{n}{2p}} \right)  e^{-\lambda_1 (t-s)}\norm{\hat{u}(s)-\tilde{u}(s)}_{L^{\frac{p}{2}}} ds\cr
&&+ k_4 \tilde{\rho} \int_0^t \left( 1 + (t-s)^{-\frac{1}{2}-\frac{n}{2p}} \right)  e^{-\lambda_1 (t-s)}\norm{\nabla(-\hat{v}(s)+\tilde{v}(s))}_{L^p} ds.
\end{eqnarray}
This implies that
\begin{eqnarray}\label{stabu}
&&e^{\sigma t}\norm{\hat{u}(t)-\tilde{u}(t)}_{L^{\frac{p}{2}}} \cr
&\leqslant& k_1e^{-(\lambda_1-\sigma) t}\norm{u_0-\tilde{u}(0)}_{L^{\frac{p}{2}}} + k_4 \rho \int_0^t \left( 1 + (t-s)^{-\frac{1}{2}-\frac{n}{2p}} \right)  e^{-(\lambda_1-\sigma) (t-s)}(e^{\sigma s}\norm{\hat{u}(s)-\tilde{u}(s)}_{L^{\frac{p}{2}}}) ds\cr
&&+ k_4 \tilde{\rho} \int_0^t \left( 1 + (t-s)^{-\frac{1}{2}-\frac{n}{2p}} \right)  e^{-(\lambda_1-\sigma) (t-s)}(e^{\sigma s}\norm{\nabla(-\hat{v}(s)+\tilde{v}(s))}_{L^p}) ds.
\end{eqnarray}
Moreover, {we also have} the following estimate
\begin{eqnarray}
\norm{\hat{v}(t)-\tilde{v}(t)}_{L^{\frac{p}{2}}} &\leqslant& \norm{e^{t(\Delta-\mathrm{Id})}(v_0-\tilde{v}(0)}_{L^{\frac{p}{2}}} + \int_0^t \norm{e^{(t-s)(\Delta-\mathrm{Id})}(\hat{u}(s)-\tilde{v}(s))}_{L^{\frac{p}{2}}}ds\cr
&\leqslant& k_1 e^{-(\lambda_1+1)t} \norm{v_0-\tilde{v}(0)}_{L^{\frac{p}{2}}} + k_1\int_0^t e^{-(\lambda_1+1)(t-s)} \norm{\hat{u}(s)-\tilde{u}(s)}_{L^{\frac{p}{2}}}ds\cr
&\leqslant& k_1 e^{-(\lambda_1+1)t} \norm{v_0-\tilde{v}(0)}_{L^{\frac{p}{2}}} + k_1^2\int_0^t e^{-(\lambda_1+1)(t-s)} e^{-\lambda_1s}\norm{u_0 - \tilde{u}(0)}_{L^{\frac{p}{2}}} ds\cr
&& + k_1k_4\rho\int_0^t e^{-(\lambda_1+1)(t-s)}\int_0^s\left(1+(s-z)^{-\frac{1}{2}-\frac{n}{2p}} \right)e^{-\lambda_1(s-z)}\norm{\hat{u}(z)-\tilde{u}(z)}_{L^{\frac{p}{2}}}dzds\cr
&& + k_1k_4\tilde{\rho}\int_0^t e^{-(\lambda_1+1)(t-s)}\int_0^s\left(1+(s-z)^{-\frac{1}{2}-\frac{n}{2p}} \right)e^{-\lambda_1(s-z)}\norm{\nabla(-\hat{v}(z)+\tilde{v}(z))}_{L^p}dzds\cr
&\leqslant& k_1 e^{-(\lambda_1+1)t} \norm{v_0-\tilde{v}(0)}_{L^{\frac{p}{2}}} + k_1^2e^{-\lambda_1t}\norm{u_0 - \tilde{u}(0)}_{L^{\frac{p}{2}}}\int_0^t e^{-(t-s)}  ds\cr
&& + k_1k_4\rho\int_0^t e^{-(t-s)}\int_0^s\left(1+(s-z)^{-\frac{1}{2}-\frac{n}{2p}} \right)e^{-\lambda_1(t-z)}\norm{\hat{u}(z)-\tilde{u}(z)}_{L^{\frac{p}{2}}}dzds\cr
&& + k_1k_4\tilde{\rho}\int_0^t e^{-(t-s)}\int_0^s\left(1+(s-z)^{-\frac{1}{2}-\frac{n}{2p}} \right)e^{-\lambda_1(t-z)}\norm{\nabla(-\hat{v}(z)+\tilde{v}(z))}_{L^p}dzds.
\end{eqnarray}
This leads to
\begin{eqnarray}\label{stabv}
&&e^{\sigma t}\norm{\hat{v}(t)-\tilde{v}(t)}_{L^{\frac{p}{2}}} \cr
&\leqslant& k_1 e^{-(\lambda_1+1-\sigma)t} \norm{v_0-\tilde{v}(0)}_{L^{\frac{p}{2}}} + k_1^2e^{-(\lambda_1-\sigma)t}\norm{u_0 - \tilde{u}(0)}_{L^{\frac{p}{2}}}\int_0^t e^{-(t-s)}  ds\cr
&& + k_1k_4\rho\int_0^t \int_0^s\left(1+(s-z)^{-\frac{1}{2}-\frac{n}{2p}} \right)e^{-(\lambda_1-\sigma)(t-z)}e^{-(t-s)}(e^{\sigma z}\norm{\hat{u}(z)-\tilde{u}(z)})_{L^{\frac{p}{2}}}dzds\cr
&& + k_1k_4\tilde{\rho}\int_0^t \int_0^s\left(1+(s-z)^{-\frac{1}{2}-\frac{n}{2p}} \right)e^{-(\lambda_1-\sigma)(t-z)}e^{-(t-s)}(e^{\sigma z}\norm{\nabla(-\hat{v}(z)+\tilde{v}(z))}_{L^p})dzds.
\end{eqnarray}
Similar as above and by using estimates \eqref{dispersive2} and \eqref{dispersive3} we have
\begin{eqnarray}\label{stabderv}
&&e^{\sigma t}\norm{\nabla( \hat{v}(t)-\tilde{v}(t))}_{L^{p}}\cr
&\leqslant& k_3 e^{-(\lambda_1+1-\sigma)t} \norm{\nabla (v_0-\tilde{v}(0))}_{L^{p}} + k_1k_2e^{-(\lambda_1-\sigma)t}\norm{u_0 - \tilde{u}(0)}_{L^{\frac{p}{2}}}\int_0^t e^{-(t-s)}  ds\cr
&&+ k_2k_4{\rho}\int_0^t\int_0^s \left( 1 + (t-z)^{-\frac{1}{2}-\frac{n}{2p}} \right)e^{-(\lambda_1-\sigma)(t-z)} e^{-(t-s)}(e^{\sigma z}\norm{\hat{u}(z)-\tilde{u}(z)}_{L^{\frac{p}{2}}})dzds\cr
&& + k_2k_4\tilde{\rho}\int_0^t \int_0^s\left(1+(s-z)^{-\frac{1}{2}-\frac{n}{2p}} \right)e^{-(\lambda_1-\sigma)(t-z)}e^{-(t-s)}(e^{\sigma z}\norm{\nabla(-\hat{v}(z)+\tilde{v}(z))}_{L^p})dzds.
\end{eqnarray}
Setting $\phi(t) = e^{\sigma t}\norm{\begin{bmatrix}\hat{u}-\tilde{u}\\\hat{v}-\tilde{v}\end{bmatrix}(t)}$. Combining inequalities \eqref{stab}, \eqref{stabu}, \eqref{stabv} and \eqref{stabderv} we obtain that
\begin{eqnarray}
\phi(t) &\leqslant& M\norm{\begin{bmatrix}u_0-\tilde{u}(0)\\v_0-\tilde{v}(0)\end{bmatrix}} + N \int_0^t \left(1+(t-s)^{-\frac{1}{2}-\frac{n}{2p}} \right)e^{-(\lambda_1-\sigma)(t-s)}\phi(s)ds \cr
&&+ P\int_0^t \int_0^s \left(1+(s-z)^{-\frac{1}{2}-\frac{n}{2p}} \right)e^{-(\lambda_1-\sigma)(t-z)}e^{-(t-s)}\phi(z)dzds.
\end{eqnarray}
By using the Gronwall-type inequality (see \cite[Theorem 1, pages 383-384]{MiPeFi}), we get
\begin{eqnarray}\label{StabIne}
\phi(t) &\leqslant& M\norm{\begin{bmatrix}u_0-\tilde{u}(0)\\v_0-\tilde{v}(0)\end{bmatrix}}
\exp \left(N \int_0^t \left(1+(t-s)^{-\frac{1}{2}-\frac{n}{2p}} \right)e^{-(\lambda_1-\sigma)(t-s)}ds \right)\cr
&&\times \exp\left( P\int_0^t\int_0^s  \left(1+(s-z)^{-\frac{1}{2}-\frac{n}{2p}} \right)e^{-(\lambda_1-\sigma)(t-z)}e^{-(t-s)}dzds \right).
\end{eqnarray}
{By straightforward calculations, we can estimate the integrals in the right-hand side of \eqref{StabIne} as follows:}
\begin{equation}\label{Stab11}
\int_0^t \left(1+(t-s)^{-\frac{1}{2}-\frac{n}{2p}} \right)e^{-(\lambda_1-\sigma)(t-s)}ds< \frac{1}{\lambda_1-\sigma} + (\lambda_1-\sigma)^{-\frac{1}{2}+\frac{n}{2p}}{\bf \Gamma}\left( \frac{1}{2} - \frac{n}{2p} \right)<+\infty
\end{equation}
and 
\begin{eqnarray}\label{Stab12}
&&\int_0^t\int_0^s  \left(1+(s-z)^{-\frac{1}{2}-\frac{n}{2p}} \right)e^{-(\lambda_1-\sigma)(t-z)}e^{-(t-s)}dzds\cr
&<& \int_0^t\int_0^u \left(1+v^{-\frac{1}{2}-\frac{n}{2p}} \right)e^{-(\lambda_1-\sigma)(u+v)}e^{-u}dvdu \hbox{   (where   }u=t-s,v=s-z)\cr
&=& \int_0^t e^{-(\lambda_1-\sigma)u}e^{-u}\int_0^u \left(1+v^{-\frac{1}{2}-\frac{n}{2p}} \right)e^{-(\lambda_1-\sigma)v}dvdu\cr
&<& \int_0^t e^{-(\lambda_1-\sigma)u}e^{-u}\int_0^{+\infty} \left(1+v^{-\frac{1}{2}-\frac{n}{2p}} \right)e^{-(\lambda_1-\sigma)v}dvdu\cr
&=& \int_0^t e^{-(\lambda_1-\sigma)u}e^{-u} \left(  \frac{1}{\lambda_1-\sigma} + (\lambda_1-\sigma)^{-\frac{1}{2}+\frac{n}{2p}}{\bf \Gamma}\left(  \frac{1}{2} -\frac{n}{2p}\right)\right)du\cr
&<&\left( \frac{1}{\lambda_1-\sigma} + (\lambda_1-\sigma)^{-\frac{1}{2}+\frac{n}{2p}}{\bf \Gamma}\left( \frac{1}{2} - \frac{n}{2p} \right) \right)  \int_0^{+\infty} e^{-(\lambda_1-\sigma+1)u}du\cr
&=& \frac{1}{\lambda_1+1-\sigma}\left( \frac{1}{\lambda_1-\sigma} + (\lambda_1-\sigma)^{-\frac{1}{2}+\frac{n}{2p}}{\bf \Gamma}\left( \frac{1}{2} - \frac{n}{2p} \right) \right) <+\infty.
\end{eqnarray}
Combining inequalities \eqref{StabIne}, \eqref{Stab11} and \eqref{Stab12}, we obtain the exponential stability \eqref{expStable}. Our proof is complete.
\end{proof}

\end{document}